\documentclass{article}

\usepackage{amssymb}
\usepackage{amsmath}
\usepackage{amsthm}
\usepackage{amsfonts}
\usepackage{graphicx}
\usepackage{ifthen}
\usepackage{mathrsfs}
\usepackage{hyperref}
\usepackage[mathcal]{euscript}
\usepackage{pgf,tikz}
\usetikzlibrary{arrows}
\usepackage{calligra}
\usepackage[nospace]{cite}
\usepackage[cmtip,arrow]{xy}
\usepackage{comment}
\usepackage{cite}
\newtheorem{Theorem}{Theorem}   [section]

\newtheorem{Construction}[Theorem] {Construction}

\newtheorem{Proposition}[Theorem] {Proposition}

\newtheorem{Ruleset}{Ruleset}

\title{The Game of Blocking Pebbles}
\author{Michael Fisher$^1$, Craig Tennenhouse$^2$\\
\small
$^1$Dept. of Mathematics, West Chester University, PA, USA\\
\small
$^2$Dept. of Mathematical Sciences, University of New England, ME, USA}

\begin{document}
\maketitle

\begin{abstract}
Graph Pebbling is a well-studied single-player game on graphs. We introduce the game of \textsc{blocking pebbles} which adapts Graph Pebbling into a two-player strategy game in order to examine it within the context of Combinatorial Game Theory. Positions with game values matching all integers, all nimbers, and many infinitesimals and switches are found. 
\end{abstract}

\section{Introduction}\label{sec:intro}
Graph Pebbling is an area of current interest in Graph Theory. In an undirected graph $G$, a root vertex $r$ is designated. Heaps of pebbles are placed on the vertices of $G$, with a legal move consisting of choosing a vertex $v$ with at least two pebbles, removing two pebbles, and placing a single pebble on a neighbor of $v$. The goal is to \emph{pebble}, or place a pebble on, the vertex $r$. The \emph{pebbling number} of $G$, denoted $\pi(G)$, is the fewest number of pebbles necessary so that any initial distribution of $\pi(G)$ pebbles among the vertices of $G$, and any vertex of $G$ chosen as the root, has a sequence of moves resulting in the root being pebbled.

Introduced by Chung in 1989~\cite{chung1989pebbling}, a number of results on pebbling of different families of graphs have been found. Of note are pebbling numbers of paths, cycles~\cite{pachter1995pebbling}, and continuing work on a conjecture of Graham's on the Cartesian products of graphs~\cite{chung1989pebbling}. Time complexity is also known, both for determination of $\pi(G)$ and for the minimum number of moves in a successful pebbling solution, for general graphs. See~\cite{hurlbert1999survey} for a survey of results in Graph Pebbling.

The results and language here are in reference to Combinatorial Game Theory (CGT). The \emph{nim sum}, also called the \emph{digital sum}, of non-negative integers is the result of their sum in binary without carry. This is denoted $x_1\oplus x_2$ if there are only two numbers, and in the case of more we use the notation $\sum \oplus x_i$. For more notation and background on the computation of CGT game values, we refer the reader to~\cite{berlekamp2003winning,albert2007lessons}.

In Sect. \ref{sec:ruleset} we will introduce a two-player combinatorial ruleset based on Graph Pebbling, with subsequent sections addressing results on both impartial and partisan positions. The widely studied game of \textsc{hackenbush} is related to the game of \textsc{blocking pebbles}, so we remind the reader of its rules here. As in all two-player combinatorial games, players are denoted Left and Right.

A position in \textsc{hackenbush} consists of a \emph{ground}, which is usually visualized as a horizontal line but can also be realized as a single vertex, along with a collection of three-edge-colored subgraphs (blue, red, and green) all connected to the ground, (see Fig. \ref{fig:hackenbush}). A move by Left consists of the removal of a single edge colored blue or green, while Right may only remove an edge colored red or green. Any remaining subgraphs disconnected from the ground are removed, and normal play continues until no edges remain and the player who removes the final edge is declared the winner.

In \textsc{hackenbush} all real numbers, nimbers, and ordinals are achievable game values.

\begin{figure}
\centering

\definecolor{cqcqcq}{rgb}{0.7529411764705882,0.7529411764705882,0.7529411764705882}
\definecolor{ffffff}{rgb}{1.,1.,1.}
\begin{tikzpicture}[line cap=round,line join=round,>=triangle 45,x=1.0cm,y=1.0cm]
\clip(-3.62,-2.46) rectangle (5.48,3.4);
\draw (-3.,-2.)-- (5.,-2.);
\draw [line width=3.6pt,color=cqcqcq] (-2.,-2.)-- (-2.,0.);
\draw [line width=3.6pt] (-2.,0.)-- (-3.,1.);
\draw (-3.,1.)-- (-2.,2.);
\draw [line width=3.6pt] (-2.,2.)-- (-1.,1.);
\draw (-1.,1.)-- (-2.,0.);
\draw [line width=3.6pt,color=cqcqcq] (-2.,2.)-- (-3.,3.);
\draw [line width=3.6pt,color=cqcqcq] (-2.,2.)-- (-2.,3.);
\draw [line width=3.6pt,color=cqcqcq] (-2.,2.)-- (-1.,3.);
\draw [shift={(3.,-0.5)},line width=3.6pt]  plot[domain=1.5707963267948966:4.71238898038469,variable=\t]({1.*1.5*cos(\t r)+0.*1.5*sin(\t r)},{0.*1.5*cos(\t r)+1.*1.5*sin(\t r)});
\draw [shift={(3.,0.)},line width=3.6pt,color=cqcqcq]  plot[domain=-1.5707963267948966:1.5707963267948966,variable=\t]({1.*1.*cos(\t r)+0.*1.*sin(\t r)},{0.*1.*cos(\t r)+1.*1.*sin(\t r)});
\draw [shift={(3.,2.)}] plot[domain=-1.5707963267948966:1.5707963267948966,variable=\t]({1.*1.*cos(\t r)+0.*1.*sin(\t r)},{0.*1.*cos(\t r)+1.*1.*sin(\t r)});
\draw [shift={(3.,2.5)},line width=3.6pt]  plot[domain=1.5707963267948966:4.71238898038469,variable=\t]({1.*0.5*cos(\t r)+0.*0.5*sin(\t r)},{0.*0.5*cos(\t r)+1.*0.5*sin(\t r)});
\begin{scriptsize}
\draw [fill=ffffff] (-2.,-2.) circle (2.5pt);
\draw [fill=ffffff] (-2.,0.) circle (2.5pt);
\draw [fill=ffffff] (-3.,1.) circle (2.5pt);
\draw [fill=ffffff] (-2.,2.) circle (2.5pt);
\draw [fill=ffffff] (-1.,1.) circle (2.5pt);
\draw [fill=ffffff] (-3.,3.) circle (2.5pt);
\draw [fill=ffffff] (-2.,3.) circle (2.5pt);
\draw [fill=ffffff] (-1.,3.) circle (2.5pt);
\draw [fill=ffffff] (3.,-2.) circle (2.5pt);
\draw [fill=ffffff] (3.,1.) circle (2.5pt);
\draw [fill=ffffff] (3.,-1.) circle (2.5pt);
\draw [fill=ffffff] (3.,3.) circle (2.5pt);
\draw [fill=ffffff] (3.,2.) circle (2.5pt);
\end{scriptsize}
\end{tikzpicture}
\caption{A \textsc{hackenbush} position}
\label{fig:hackenbush}
\end{figure}
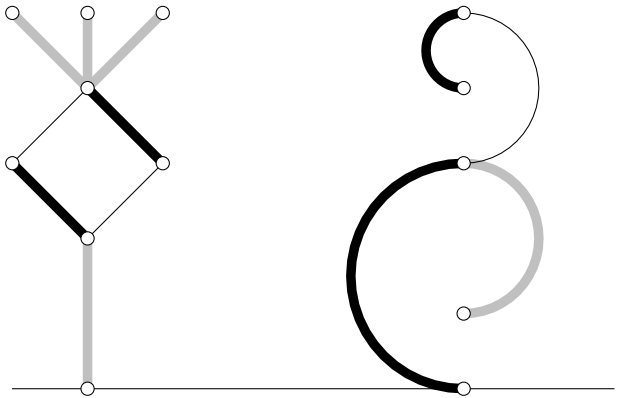

\section{Ruleset and play}\label{sec:ruleset}
A game of \textsc{blocking pebbles} consists of a directed acyclic graph $G$ and a $3$-tuple $(b,r,g)$ at each vertex of $G$, representing the number of blue, red, and green pebbles. Left may move blue and green pebbles, while Right may move red and green. This follows the convention of \textsc{hackenbush} wherein players may remove an edge of their own color or the neutral color green. However, where \textsc{hackenbush} players may only choose a single edge, in \textsc{blocking pebbles} players may move any number of pebbles at a single vertex. In this way, \textsc{blocking pebbles} is similar to \textsc{graph nim}~\cite{calkin2010computing}.

\begin{Ruleset}\label{rules}
Given a tuple of the form $(b,r,g)$ at each vertex of a directed acyclic graph $G$, Left can make one of the following two moves from the vertex $v$.
\begin{enumerate}
    \item Move a positive number of blue and/or green pebbles from $v$ to an in-neighbor of $v$
    \item Remove two blue and/or green pebbles from $v$ and place one on an out-neighbor of $v$
\end{enumerate}
No blue pebbles can be moved to a vertex with a non-zero number of red pebbles. Right has the obvious symmetric moves.

Play proceeds following the normal play convention, where the last player to make a legal move wins.
\end{Ruleset}

Note that if Left removes one blue and one green pebble from $v$ she may add the green to $v$'s out-neighbor. However, it is always preferable to instead add the blue as this results in a position with more blue pebbles and increases the number of vertices blocked by Left.

\begin{figure}\label{fig:ex1}
\centering
\small
\definecolor{qqqqff}{rgb}{0.,0.,1.}
\begin{tikzpicture}[line cap=round,line join=round,>=triangle 45,x=1.0cm,y=1.0cm,scale=0.65]
\clip(-10.06,-5.34) rectangle (9.5,4.98);
\draw [->] (-2.,4.) -- (-4.,1.);
\draw [->] (-4.,1.) -- (0.,1.);
\draw [->] (-2.,4.) -- (0.,1.);
\draw [->] (0.,1.) -- (3.,1.);
\draw (-2.48,4.8) node[anchor=north west] {(2,0,0)};
\draw (-4.48,0.74) node[anchor=north west] {(0,1,1)};
\draw (-0.48,0.74) node[anchor=north west] {(0,0,0)};
\draw (2.52,0.76) node[anchor=north west] {(0,1,1)};
\draw [->] (-7.,-1.) -- (-9.,-4.);
\draw [->] (-9.,-4.) -- (-5.,-4.);
\draw [->] (-7.,-1.) -- (-5.,-4.);
\draw [->] (-5.,-4.) -- (-2.,-4.);
\draw (-7.48,-0.2) node[anchor=north west] {(0,0,0)};
\draw (-9.48,-4.26) node[anchor=north west] {(0,1,1)};
\draw (-5.48,-4.26) node[anchor=north west] {(1,0,0)};
\draw (-2.48,-4.24) node[anchor=north west] {(0,1,1)};
\draw [->] (3.,-1.) -- (1.,-4.);
\draw [->] (1.,-4.) -- (5.,-4.);
\draw [->] (3.,-1.) -- (5.,-4.);
\draw [->] (5.,-4.) -- (8.,-4.);
\draw (2.52,-0.2) node[anchor=north west] {(2,0,0)};
\draw (0.52,-4.26) node[anchor=north west] {(0,1,1)};
\draw (4.52,-4.26) node[anchor=north west] {(0,0,1)};
\draw (7.52,-4.24) node[anchor=north west] {(0,1,0)};
\begin{scriptsize}
\draw [fill=qqqqff] (-2.,4.) circle (2.5pt);
\draw[color=qqqqff] (-1.54,3.98) node {$A$};
\draw [fill=qqqqff] (-4.,1.) circle (2.5pt);
\draw[color=qqqqff] (-4.06,1.62) node {$B$};
\draw [fill=qqqqff] (0.,1.) circle (2.5pt);
\draw[color=qqqqff] (-0.04,1.62) node {$C$};
\draw [fill=qqqqff] (3.,1.) circle (2.5pt);
\draw[color=qqqqff] (2.92,1.56) node {$D$};
\draw [fill=qqqqff] (-7.,-1.) circle (2.5pt);
\draw [fill=qqqqff] (-9.,-4.) circle (2.5pt);
\draw [fill=qqqqff] (-5.,-4.) circle (2.5pt);
\draw [fill=qqqqff] (-2.,-4.) circle (2.5pt);
\draw [fill=qqqqff] (3.,-1.) circle (2.5pt);
\draw [fill=qqqqff] (1.,-4.) circle (2.5pt);
\draw [fill=qqqqff] (5.,-4.) circle (2.5pt);
\draw [fill=qqqqff] (8.,-4.) circle (2.5pt);
\end{scriptsize}
\end{tikzpicture}
\normalsize
\caption{A position in \textsc{blocking pebbles} and two of Left's options}
\end{figure}
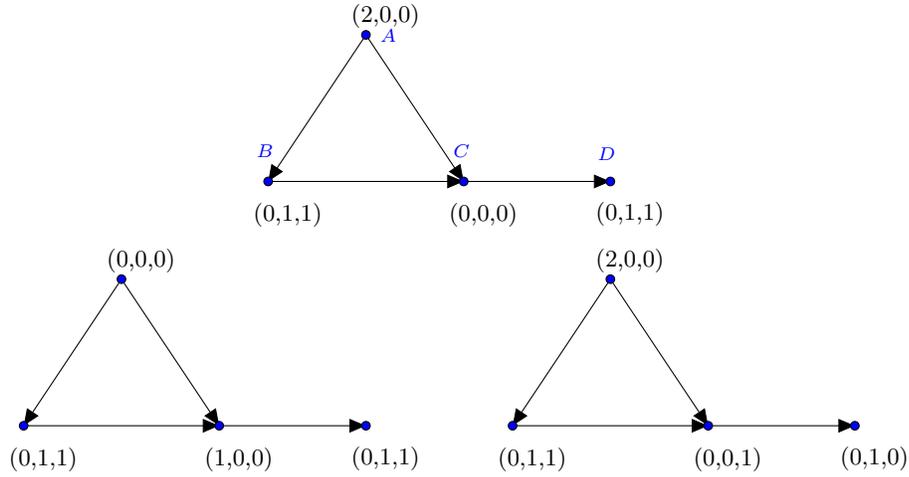

As an example, consider the position in Figure \ref{fig:ex1}. At the top is a position in \textsc{blocking pebbles}. Note that Left cannot move any blue pebbles from vertex $A$ to $B$ since $B$ already contains a red pebble. However, Left can move a single blue pebble from $A$ to $C$ at a cost of one blue pebble. She can also move the one green pebble from $D$ to $C$.

An interesting property of this ruleset is the existence of \emph{discovered moves}, similar to discovered attacks in chess. A player may be unable to move at one point in the game, but after their opponent moves then the game is once again playable by the first player. As an example consider a simple out-star with two red pebbles on the source and a single blue pebble on a sink node. Left has no moves, but once Right moves Left can move their pebble to the source.

\section{Blue-Red-Green Blocking Pebbles}\label{sec:br}

In this section we will address some families of game values that are achievable in \textsc{blocking pebbles}. We will only address finite graphs, hence we will not encounter non-dyadic rationals. This is similar to hackenbush, described in Sect.~\ref{sec:intro}. Due to the complexity of analysis, we will also restrict our graphs to orientations of stars, paths, and small graphs.

We begin with a simple result.
\begin{Theorem}\label{thm:integers}
For every $k\in \mathbb{Z}$ there is a position in \textsc{blocking pebbles} with value $k$.
\end{Theorem}
\begin{proof}
Let $G$ be a single arc directed from $u$ to $v$. If $k>0$ then place $2k$ blue pebbles and a single red pebble on $u$. Switch red and blue pebbles if $k<0$. Zero is trivially achieved by a graph with no pebbles, or any number of other pebble distributions.
\end{proof}

Regarding infinitesimals, $\downarrow$ is realized by an out-star with two leaves; that is, a vertex $u$ with out-neighbors $v_1,v_2$. Vertex $v_1$ has a blue pebble, and $v_2$ has one red and one green pebble. Left can move the blue or green pebble to $u$, which is simple to identify as $*$. Right, however, can move the green to the source vertex $u$ resulting in $*$, the red to $u$ resulting in zero, or both red and green pebbles to $u$ which is also a zero position. Since zero dominates $*$, the initial position is $\{*|0\}=\downarrow$.

Due to the blocking rule, \textsc{blocking pebbles} is relatively unique among partisan combinatorial games. In \textsc{hackenbush} the presence of a move for one player does not inhibit moves for the other. In \textsc{clobber}, another two-player partisan combinatorial game, (see~\cite{albert2005introduction}), the presence of a red piece actually encourages movement for Left, and vice versa. This is a property common to all dicot games. However, in \textsc{blocking pebbles} a single well-placed blue pebble, for example, can cut off many of Right's moves. It's natural, then, that many positions result in game values that are switches.

Other results concerning bLue/Red Out-Star positions include the following.

\bigskip

\begin{Theorem} Let $b_c \geq 0$  be the number of blue pebbles on the non-leaf vertex (center vertex) of a given out-star $T$.  Further, let $b_{\ell} \geq 0$ be the total number of blue pebbles on the leaves of $T$.  The parameters $r_c$ and $r_{\ell}$ are defined similarly for red. Then 

\begin{enumerate}

\item if $T$ has only blue pebbles, then the game value of $T$ is $3b_{\ell} - 2 + 2b_c$,


\item if $T$ has at least one blue pebble and no red pebbles on its center, at least one blue pebble on some collection of leaves, and at least one red pebble on some other collection of leaves, the game value of $T$ is $3b_{\ell} - 2 + 2(b_c-1)$, 

\item if $T$ has no pebbles on its center, exactly one blue pebble on a leaf, and exactly one red pebble on a leaf (possibly the same leaf), the game value of $T$ is $*$,

\item if $T$ has no pebbles on its center, two or more blue pebbles on some collection of leaves, and exactly one red pebble on some other leaf, the game value of $T$ is $\{ 3(b_{\ell}-1)-2 \ | \ 0 \}$, and 

\item if $T$ has no pebbles on its center, at least two blue pebbles on some collection of leaves, and at least two red pebbles on some other collection of leaves, the game value of $T$ is $$\{3(b_{\ell}-1)-2 \ | \ -(3(r_{\ell}-1)-2)\}.$$

\end{enumerate}

\end{Theorem}

\bigskip

\begin{proof}  For (1), moving $j$ blue pebbles from a leaf to the center vertex gives a position of value $3(b_{\ell}-j) - 2 + 2(b_c + j) =  3b_{\ell} - 2 + 2b_c - j$.  Moving a pebble from the center to a leaf gives a position of value $3(b_{\ell}+1) - 2 + 2(b_c - 2) =  3b_{\ell} - 2 + 2b_c - 1$.  Thus, the Left option obtained by moving one blue center pebble to a leaf dominates (or is equal to) every other Left option.  As Right has no available move, by induction, we see that the value of any position of this type is $$\{  3b_{\ell} - 2 + 2b_c - 1 \ | \ \} = 3b_{\ell} - 2 + 2b_c.$$

\bigskip

In (2), like in (1), we see that an optimal move for Left is to move a blue pebble to the central vertex.  As Right again has no available move, we see that the value of any position of this type is $$\{  3b_{\ell} - 2 + 2(b_c-1) - 1 \ | \ \} = 3b_{\ell} - 2 + 2(b_c-1).$$
\bigskip

Case (3) is easily seen to be true.
\bigskip

In case (4), we see by using past reasoning that Left's best option is to move exactly one pebble to the center, yielding a position of value $3(b_{\ell}-1) - 2.$ Right's only move is to 0.   Hence, the value of the starting position is $\{ 3(b_{\ell}-1)-2 \ | \ 0 \}$.
\bigskip

Case (5) now follows easily from prior reasoning.  

\end{proof}

\bigskip

For the next result, we use the following notation for a bLue/Red pebbling configuration of the out-star $K_{1,2}$:  $[(a,b),[c,d],[e,f]]$ is the configuration with $a$ blue pebbles and $b$ red pebbles on the central vertex, $c$ blue pebbles and $d$ red pebbles on one of the pendant vertices, and $e$ blue pebbles and $f$ red pebbles on the other pendant vertex.

\begin{Theorem} The following results pertain to a given Pebbling configuration on the out-star $K_{1,2}$. 

\begin{enumerate}

\item For $c \geq 1$, the position $[(a,b),[0,c],[0,0]]$ has value $-\lfloor \frac{b}{2} \rfloor$ if $a = 1$ and value $\left\{ \lfloor \frac{a}{2}  \rfloor - 1 \ \vert \ \lfloor \frac{a-b}{2} \rfloor + 1 \right\}$ if $a \geq 2$,

\item for $a, b, c, d \geq 1$, the position $[(a,b),[c,0],[0,d]]$ has value $\lfloor \frac{a-b}{2} \rfloor$,

\item for $a, b, c, d, e \geq 1$, the position $[(a,b),[c,0],[d,e]]$ has value $\lfloor \frac{a-1}{2} \rfloor$,

\item for $a, b, c, d \geq 1$, the position $[(0,0),[a,b],[c,d]]$ has value $\{ a+c-1 \ \vert \ -(b+d-1) \}$,

\item for $a, b, c \geq 1$, the position $[(0,0),[a,b],[0,c]]$ has value $\{ a-1 \ \vert \ -(3(b+c)-5) \}$,

\item for $a, b \geq 1$, the position $[(0,0),[a,b],[0,0]]$ has value $\{ 3a-5 \ \vert \ -(3b-5) \}$,

\item for $b \geq 1$, $[(1,0),[0,b],[0,0]]$ and $[(2,0),[0,b],[0,0]]$ are both zero positions.

\end{enumerate}

\end{Theorem}

\bigskip

\begin{proof} For (1), the position $[(1,1),[c,0],[0,0]]$ is the zero position.  It is also readily checked that the position $[(1,2),[c,0],[0,0]]$ has value 0.

If $b > 2$, then Left has no move from $[(1,b),[c,0],[0,0]]$. From $[(1,b),[c,0],[0,0]]$, Right may move to the position $[(1,b-2),[c,0],[0,1]]$, which has value $\lfloor \frac{-b+1}{2} \rfloor = -\lfloor \frac{b}{2} \rfloor + 1$ by induction.  Hence, $[(1,b),[c,0],[0,0]]$ has value $-\lfloor \frac{b}{2} \rfloor$ as required.

If $a \geq 2$, then Left's best move from $[(a,b),[c,0],[0,0]]$ is to $[(a-2,b),[c,0],[1,0]]$ which has value $\lfloor \frac{a-2}{2} \rfloor$ (Right has no move from this position and Left has $\lfloor \frac{a-2}{2} \rfloor$ moves).  Right's only move is to $[(a,b-2),[c,0],[0,1]]$ which has value $\lfloor \frac{a-b+2}{2} \rfloor$, also by induction.  Hence, $[(a,b),[c,0],[0,0]]$ has value $\left\{ \lfloor \frac{a}{2}  \rfloor - 1 \ \vert \ \lfloor \frac{a-b}{2} \rfloor + 1 \right\}$ when $a \geq 2$.

For (2),  it is clear that the position $[(1,1),[c,0],[0,d]]$ is a zero position.  If $a \geq 2$, then, from $[(a,1),[c,0],[0,d]]$, Left has a move to $[(a-2,1),[c+1,0],[0,d]]$ and Right has no move.  Thus, $[(a,1),[c,0],[0,d]]$ has value $\lfloor \frac{a-1}{2} \rfloor$, by induction.  A similar argument establishes the claim that $[(1,b),[c,0],[0,d]]$ has value $\lfloor \frac{1-b}{2} \rfloor$.

Now if $a, b \geq 2$, then, from $[(a,b),[c,0],[0,d]]$, Left has the move to $[(a-2,b),[c+1,0],[0,d]]$ and Right has the move to $[(a,b-2),[c,0],[0,d+1]]$.  By induction, we see that $[(a,b),[c,0],[0,d]]$ has value $$\left\{ \lfloor \frac{a-2-b}{2} \rfloor \vert \lfloor \frac{a-b+2}{2} \rfloor \right\} = \left\{ \lfloor \frac{a-b}{2} \rfloor -1\vert \lfloor \frac{a-b}{2} \rfloor+1 \right\}  = \lfloor \frac{a-b}{2} \rfloor.$$

For case (3), note that if $a=1$, then there are no moves for either player; the formula given correctly yields the game value 0.  If $a=2$, then from the position $[(2,b),[c,0],[d,e]]$ Left has the move to $[(0,b),[c+1,0],[d,e]]$.  From here, Left has no move and Right has $e$ moves.  Thus the position $[(0,b),[c+1,0],[d,e]]$ has value $-e$.  Hence, $[(2,b),[c,0],[d,e]]$ has value 0, as required.

If $a > 2$, then, from $[(a,b),[c,0],[d,e]]$, Left can move to $[(a-2,b),[c+1,0],[d,e]]$ which has value $\lfloor\frac{a-3}{2} \rfloor = \lfloor\frac{a-1}{2} \rfloor - 1$, by induction.  Right has no moves from $[(a,b),[c,0],[d,e]]$.  Hence, $[(a,b),[c,0],[d,e]]$ has value  $\{\lfloor \frac{a-1}{2} \rfloor - 1 \vert \ \} = \lfloor \frac{a-1}{2}\rfloor$ as desired.

For (4), note that Left's only move from $[(0,0),[a,b],[c,d]]$ is to $[(1,0),[a-1,b],[c,d]]$.  This last position has value $a+c-1$ by induction.  Similarly, Right's only move from $[(0,0),[a,b],[c,d]]$ is to $[(0,1),[a,b-1],[c,d]]$.  This position has value $-(b+d-1)$ by induction.  It now follows that $[(0,0),[a,b],[c,d]]$ has value $\{ a+c-1 \ \vert \ -(b+d-1) \}$.

Cases (5) and (6) follow from the previous result and case (7) is trivial.

\end{proof}

\bigskip

Below we present some results pertaining to given positions on an in-star $T$.

\bigskip

\begin{Theorem} Let $b_c \geq 0$  be the number of blue pebbles on the non-leaf vertex (center vertex) of a given in-star $T$.  Further, let $b_{\ell} \geq 0$ be the total number of blue pebbles on the leaves of $T$.  The parameters $r_c$ and $r_{\ell}$ are defined similarly for red. Then 

\begin{enumerate}

\item if $T$ has at least one blue pebble on some collection of leaves, at least one blue pebble on the center vertex, but no red pebbles anywhere, the game value of $T$ is $3b_{c} + 2b_{\ell} - 2$,

\item the position $[(0,0),[0,1],[1,0]]$ has value  0 and the position $[(0,0),[0,2],[2,0]]$ has value  $*$,

\item for $a \geq 2$, the position $[(0,0),[a,0],[0,1]]$ has value  $2a-2$,

\item for $a \geq 3$, the position $[(0,0),[a,0],[0,2]]$ has value $\{2a-6 | 0 \}$,

\item for $a, b \geq 3$, the position $[(0,0),[a,0],[0,b]]$ has value $\{2a-6 | -2b+6 \}$, and

\item for $a, b \geq 1$ and $c \geq 2$, the position $[(a,0),[b,0],[0,c]]$ has value $3a+2b-5$.

\end{enumerate}

\end{Theorem}

\bigskip

\begin{proof} For (1), moving a single pebble from the center vertex to a leaf is one of the two best moves for Left;  Right has no move.  By induction, this position has value $\{3b_{c} + 2b_{\ell} - 3 | \ \} = 3b_{c} + 2b_{\ell} - 2$ as required.

For the first half of case (2), observe that no player can move from $[(0,0),[0,1],[1,0]]$; hence, the game value is 0. The proofs of the remaining cases are by simultaneous induction. 

For the position  $[(0,0),[0,2],[2,0]]$ in case (2), Left's only move is to $[(1,0),[0,2],[0,0]]$.  Right has no move from this position, but Left can move to $[(0,0),[1,0],[0,2]]$.  From here, Left has no move and Right can move to $[(0,1),[1,0],[0,0]]$.  It is easy to check that this position has value $-1$. Hence, $[(0,0),[1,0],[0,2]]$ has value $-2$, and thus $[(1,0),[0,2],[0,0]]$ has value $0$.  By symmetry $[(0,1),[0,0],[2,0]]$ also has value $0$.  Therefore, $[(0,0),[0,2],[2,0]]$ has value $*$.

Case (3) can be regarded as a case belonging to case (1) as Right has no move.  Hence, $[(0,0),[a,0],[0,1]]$ has value $2a-2$ for $a \geq 2$.

In case (4), Left can move the position $[(0,0),[a,0],[0,2]]$ to $[(1,0),[a-2,0],[0,2]]$, which has value $3+2(a-2) - 5 = 2a - 6$, by induction and case (6).  Right can move the position $[(0,0),[a,0],[0,2]]$ to $[(0,1),[a,0],[0,0]]$.  It is easy to check that the last position has value $0$.   Therefore, for $a \geq 3$, the position $[(0,0),[a,0],[0,2]]$ has value $\{2a-6 | 0 \}$.

Note that (5) follows by induction and case (4).  Similarly, (6) follows by induction and cases (4) and (5).   

\end{proof}

\bigskip

The next result deals with bLue/Red Pebbling on a directed path.


\bigskip

\begin{Theorem} Each of the following describes the game value for the given pebbling distribution on $P_3$ with edges directed from left-to-right.
\begin{enumerate}
\item For $a \geq 0$ and  $b, c\geq 1$, the position $[[a,0],[b,0],[0,c]]$ has value $0$, if $a = 0$ and $b = 1$ and has value $2a+3b - 5$, otherwise.

\item For $a, b \geq 1$ and $c \geq 0$, the position $[[a,0],[0,b],[0,c]]$ has value $0$, if $b = 1$ and $c = 0$, has value $-2b-3c+2$, if $a=1$, and has value $-2b-3c+4$, otherwise.

\item For $b \geq a \geq 1$, the position $[[a,0],[0,0],[0,b]]$ has value $-3b+2$, if $a = 1$, has value $\{ 0|-3b+5 \}$, if $a = 2$, and has value $\{ 2a-6 |-3b+5 \}$, otherwise.

\end{enumerate}

\end{Theorem}

\bigskip

\begin{proof} All claims will be proven simultaneously using induction (on the height of the game tree).  We start with (1) and the special case where $a=0$ and $b=1$.  Right has no move from $[[0,0],[1,0],[0,c]]$, but Left does.  Left can move to $[[1,0],[0,0],[0,c]]$.  From this position, Left now has no move (no matter what Right does) .  A bit of thought tells us that Right's best move is to $[[1,0],[0,1],[0,c-1]]$.  Right can follow this move up with $c-1$ other moves of this type, yielding the position $[[1,0],[0,c],[0,0]]$.  Right's next move is to $[[1,0],[0,c-2],[0,1]]$.  By induction (using (2)), this position has value $-2c+3$.  This then implies that $[[1,0],[0,0],[0,c]]$ has value $-3c+3$.  Hence, $[[0,0],[1,0],[0,c]]$ has value $0$ as desired.

We now show that $[[a,0],[b,0],[0,c]]$ has value $2a+3b-5$.  Note that Right has no move from this position.  Left's best moves are to $[[a+1,0],[b-1,0],[0,c]]$ and to $[[a-2,0],[b+1,0],[0,c]]$, depending on whether or not $a \geq 2$.  By induction, both $[[a+1,0],[b-1,0],[0,c]]$ and $[[a-2,0],[b+1,0],[0,c]]$ have value $2a+3b-6$.  Hence, the starting position, $[[a,0],[b,0],[0,c]]$, has value $2a+3b-5$.

\bigskip

We now prove (2). If $b=1$ and $c=0$, then neither Left nor Right has a move from $[[a,0],[0,b],[0,c]]$.  If $b \geq 2$ or $c \geq 1$, then Left has no move from $[[a,0],[0,b],[0,c]]$, but Right does.  If $a = 1$ and $c \geq 1$, then Right's best move is from $[[1,0],[0,b],[0,c]]$ to $[[1,0],[0,b+1],[0,c-1]]$.  By induction, $[[1,0],[0,b+1],[0,c-1]]$ has value $-2b-3c +3$.  Hence, $[[1,0],[0,b],[0,c]]$ has value $-2b-3c+2$.

If $a = 1$, $b \geq 2$, and $c=0$, then Right can move from $[[1,0],[0,b],[0,0]]$ to $[[1,0],[0,b-2],[0,1]]$.  The last position has value $-1$, if $b = 2$ and has value $-2(b-2) - 3 + 2 = -2b + 3$, if $b > 2$.  Thus, $[[1,0],[0,b],[0,0]]$ has value $-2$ or $-2b + 2$, respectively.

If $a > 1$, $b=1$, and $c \geq 1$, then Right can move from $[[a,0],[0,1],[0,c]]$ to $[[a,0],[0,2],[0,c-1]]$. Continuing with reasoning similar to the above (for $a = 1$ and $c \geq 1$), we see that the position $[[a,0],[0,2],[0,c-1]]$ has value $-3c+3$.  Thus, $[[a,0],[0,1],[0,c]]$ has value $-3c+2$, as Left has no option from this position.

The case $a > 1$, $b \geq 2$ is similar to other calculations done above.

\bigskip  

Let us now focus on (3).  When $a = 1$, we see that Left has no move from $[[1,0],[0,0],[0,b]]$.  Right's best move is to $[[1,0],[0,1],[0,b-1]]$.  If $b = 1$, then the last position has value 0.  Hence, $[[1,0],[0,0],[0,b]]$ has value $-1$.  If $b > 1$, then the position $[[1,0],[0,1],[0,b-1]]$ has value $-3b+3$ by induction.  Thus, $[[1,0],[0,0],[0,b]]$ has value $-3b+2$ as required.


When $a = 2$ and $b \geq 2$, both players have a move from the position $[[2,0],[0,0],[0,b]]$.  Left's only move is to $[[0,0],[1,0],[0,b]]$. This position has value $0$ by case (1).  Right's best move from $[[2,0],[0,0],[0,b]]$ is to $[[2,0],[0,1],[0,b-1]]$.  By induction, $[[2,0],[0,1],[0,b-1]]$ has value $-3b+5$.  Thus $[[2,0],[0,0],[0,b]]$ has value $\{ 0 | -3b+5 \}$.

The last case we need to look at is $a \geq 3$.  The only Left move from $[[a,0],[0,0],[0,b]]$ is to $[[a-2,0],[1,0],[0,b]]$.  By induction, this position has value $2a - 6$.  Among Right's many moves, the move to $[[a,0],[0,1],[0,b-1]]$ is the best.  By induction, this position has value $-3b+5$.  Hence, $[[a,0],[0,0],[0,b]]$ has value $\{ 2a-6 |-3b+5 \}$. 

\end{proof}

We end this section with a short discussion of the differences between \textsc{blocking pebbles} and \textsc{hackenbush}.

As noted above, the blocking mechanic of \textsc{blocking pebbles} results in a preponderance of switches, while \textsc{hackenbush} has no such positions. Also, while we would be surprised to find a dyadic that is not the game value for some \textsc{blocking pebbles} position, we have found it difficult to find even computationally. \textsc{hackenbush} positions, on the other hand, are easily constructed that have rational non-integer game values.

\section{Green-only games}\label{sec:green}

The game of \textsc{Blocking Pebbles} restricted to green pebbles is an impartial game, with positions admitting only nimbers as game values. The interested reader will seek out~\cite{berlekamp2003winning,albert2007lessons} for more on Sprague-Grundy Theory and nimbers. While there is no use for players to employ a blocking strategy, the game remains mathematically interesting for its connections to its roots in Graph Pebbling. 

First we consider in-stars and out-stars, with green pebble distributions denoted by $>g_0,g_1,\ldots ,g_n<$ and $<g_0,g_1,\ldots ,g_n>$ respectively. In each case $g_i\geq 0$ and $g_0$ is the number of pebbles on the center vertex.

\begin{Theorem}\label{thm:greeninstar}
The value of an in-star with distribution $>g_0,g_1,\ldots ,g_n<$ is $*g_0$.
\end{Theorem}
\begin{proof}
We will demonstrate this using induction on $g_0$. First note that if $g_0=0$ then any move of a green pebble to the center from a leaf, resulting in the loss of a pebble, can be countered by returning it to the same leaf. Next we note that any move from $>g_0,g_1,\ldots ,g_n<$ results in a change to $g_0$, and that there is a move from this position that results in any number of pebbles on the center node strictly less than $g_0$. Hence the in-star is equivalent to a \textsc{nim} heap of size $g_0$.
\end{proof}

The fact revealed in Thm. \ref{thm:greeninstar} that an in-star is equivalent to a single \textsc{nim} heap can be generalized to multiple heaps with an out-star.

\begin{Theorem}\label{thm:greenoutstar}
The value of an out-star with distribution $<g_0,g_1,\ldots ,g_n>$ is $*(\sum \limits_{i=1}^n \oplus g_i)$. That is, the nim sum of all even heaps.
\end{Theorem}
\begin{proof}
We note that this game is analogous to \textsc{nim}, except instead of removing pebbles from a heap they are moved to the center at no cost. The player with the advantage simply plays the winning \textsc{Nim} strategy. Any move of a pebble from the center vertex to a lead can immediately be reversed, at a net cost of one pebble from the center. Thus the number of pebbles at the center do not contribute to the game value, which equals the nim sum of the leaf heaps.
\end{proof}

On a path we get a similar result. 

\begin{Theorem}\label{thm:greenpath}
If $(g_1,\ldots ,g_n)$ is a distribution of green pebbles along a path directed left to right, then the game value is $*(\sum \oplus g_{2k})$.
\end{Theorem}
\begin{proof}
An empty path is trivial, so let's assume the claim is false and consider the set $C$ of all counter-examples with the fewest total number of pebbles. From $C$ let $(g_{0_1},g_{0_2},\ldots ,g_{0_n})$ be the last when ordered lexicographically. Any move from this position either decreases the total number of pebbles, or increases its lexicographic position. Therefore all options of $(g_{0_1},g_{0_2},\ldots ,g_{0_n})$ are outside $C$ and hence the claim holds for them. Since each has a digital sum of even terms that differs from $(\sum \limits_{\oplus} g_{2k})$, and all smaller sums are realized through \textsc{nim} moves on the even heaps, we see that $(g_{0_1},g_{0_2},\ldots ,g_{0_n})$ also satisfies the claim. Therefore, $C$ is empty and the claim is true.
\end{proof}

Note that in Theorems \ref{thm:greeninstar}, \ref{thm:greenoutstar}, and \ref{thm:greenpath}, the strategy is equivalent to \textsc{nim}. In fact, in these particular cases \textsc{blocking pebbles} is very similar to the game of \textsc{poker nim}, wherein players make \textsc{nim} moves but retain any removed pebbles, and may add them to a heap instead of removing. While \textsc{poker nim} is loopy and \textsc{blocking pebbles} is not, both games played optimally have the same strategy and the same reciprocal moves for non-\textsc{nim} moves.

We now introduce a reduction formula for all trees, which can be applied to the three previous results.

\begin{Construction}\label{const:reduction}
Let $T$ be any oriented tree with a given distribution of green pebbles, let $S$ be its set of source vertices, and let $O$ be the set of vertices of $T$ reachable by an odd length directed path from some vertex in $S$. Additionally, for a given subset $W$ of verties, let $p(W)$ be the combined total number of pebbles on $W$. 

We construct the following digraph $D(T)$ from $T$ as follows:
\begin{enumerate}
    \item $V(D(T))=\{\sigma\}\cup O$, where $O$ has the same pebbling distribution as it does in $T$ and $\sigma$ is an vertex with no pebbles.
    \item $E(D(T))=E(O)\cup \{\sigma \rightarrow \theta | \theta\in O\}$
\end{enumerate}
\end{Construction}

\begin{figure}
\centering

\definecolor{ffffff}{rgb}{1.,1.,1.}
\begin{tikzpicture}[line cap=round,line join=round,>=triangle 45,x=1.0cm,y=1.0cm]
\clip(-3.44,-0.5) rectangle (7.3,3.42);
\draw [->] (-3.,2.) -- (-1.,1.);
\draw [->] (-3.,0.) -- (-1.,1.);
\draw [->] (-1.,1.) -- (1.,1.);
\draw [->] (1.,1.) -- (3.,1.);
\draw [->] (3.,1.) -- (5.,2.);
\draw [->] (3.,1.) -- (5.,0.);
\draw [->] (5.,2.) -- (7.,2.);
\draw [->] (0.,3.) -- (1.,1.);
\draw [->] (3.,1.) -- (5.,1.);
\draw [->] (-1.,1.) -- (1.,0.);
\begin{scriptsize}
\draw [fill=ffffff] (-3.,2.) circle (2.5pt);
\draw [fill=ffffff] (-1.,1.) circle (2.5pt);
\draw [fill=ffffff] (-3.,0.) circle (2.5pt);
\draw [fill=ffffff] (1.,1.) circle (2.5pt);
\draw [fill=ffffff] (3.,1.) circle (2.5pt);
\draw [fill=ffffff] (5.,2.) circle (2.5pt);
\draw [fill=ffffff] (5.,0.) circle (2.5pt);
\draw [fill=ffffff] (7.,2.) circle (2.5pt);
\draw [fill=ffffff] (0.,3.) circle (2.5pt);
\draw [fill=ffffff] (5.,1.) circle (2.5pt);
\draw [fill=ffffff] (1.,0.) circle (2.5pt);
\end{scriptsize}
\end{tikzpicture}

\definecolor{ffffff}{rgb}{1.,1.,1.}
\begin{tikzpicture}[line cap=round,line join=round,>=triangle 45,x=1.0cm,y=1.0cm]
\clip(-3.44,-0.5) rectangle (7.3,3.42);
\draw [->] (-1.,1.) -- (1.,1.);
\draw [->] (1.,1.) -- (3.,1.);
\draw [->] (3.,1.) -- (5.,2.);
\draw [->] (3.,1.) -- (5.,0.);
\draw [->] (5.,2.) -- (7.,2.);
\draw [->] (3.,1.) -- (5.,1.);
\draw [->] (2.1,3.06) -- (-1.,1.);
\draw [->] (2.1,3.06) -- (1.,1.);
\draw [->] (2.1,3.06) -- (3.,1.);
\draw [->] (2.1,3.06) -- (5.,2.);
\draw [->] (2.1,3.06) -- (7.,2.);
\draw [->] (2.1,3.06) -- (5.,1.);
\draw [->] (2.1,3.06) -- (5.,0.);
\begin{scriptsize}
\draw [fill=ffffff] (-1.,1.) circle (2.5pt);
\draw [fill=ffffff] (1.,1.) circle (2.5pt);
\draw [fill=ffffff] (3.,1.) circle (2.5pt);
\draw [fill=ffffff] (5.,2.) circle (2.5pt);
\draw [fill=ffffff] (5.,0.) circle (2.5pt);
\draw [fill=ffffff] (7.,2.) circle (2.5pt);
\draw [fill=ffffff] (5.,1.) circle (2.5pt);
\draw [fill=ffffff] (2.1,3.06) circle (4.5pt);
\end{scriptsize}
\end{tikzpicture}

\caption{An oriented tree $T$ and the resulting graph $D(T)$ from Const. \ref{const:reduction}}
\end{figure}
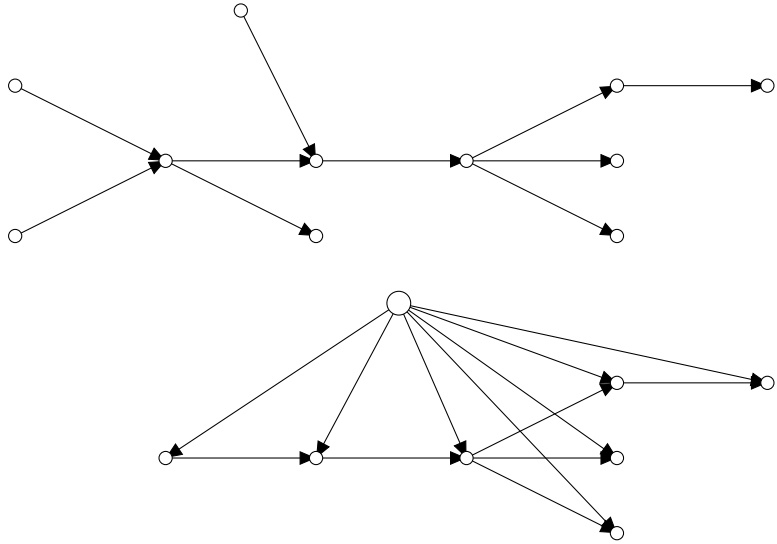

\begin{Proposition}\label{prop:reduction}
The game value of \textsc{blocking pebbles} on $T$ is equal to the game value on $D(T)$.
\end{Proposition}
\begin{proof}
The key observation is that the pebbling games on $T$ and $D(T)$ are both equivalent to \textsc{poker nim} on the set $O$. Since the two games have the same set of \textsc{nim} moves, their game values are equivalent.
\end{proof}

Applying Const. \ref{const:reduction} to an in-star results in a single arc, and when $T$ is a directed path as in Thm. \ref{thm:greenpath}, $D(T)$ is simply an out-star. Thus, many tree positions can be reduced to positions on fewer vertices.

It is worth noting, however, that many trees will not reduce to simple positions.

In particular, the transitive triple graph has proven very difficult to analyze. However, we present here the set of $\mathcal{P}$-positions.

\begin{Theorem}\label{thm:green:tt}
A position in \textsc{blocking pebbles} on a transitive triple with $g_1$ green pebbles on the source vertex, $g_3$ on the sink, and $g_2$ on the remaining vertex, is a $\mathcal{P}$-position if and only if $g_2=g_3$.
\end{Theorem}
\begin{proof}
Note that, as in all other green-only positions, pebbles on the source vertex are superfluous. Since any move that increases the total $g_2+g_3$ can be undone, we can consider these heaps as \textsc{nim} heaps and play accordingly.
\end{proof}

We close this section with a very simple result, but one that may prove useful in future investigations into the game.

\begin{Theorem}\label{thm:transtourn}
A single green pebble on the sink node of a transitive tournament on $n$ vertices is equivalent to a \textsc{nim} heap of size $n$.
\end{Theorem}
\begin{proof}
We simply consider all options of this position. Since the pebble can only move back, and can move to any previous node, this is equivalent to removing any number of stones from a \textsc{nim} heap.
\end{proof}
\section{Further directions}

There remain many open questions and avenues for further study of \textsc{blocking pebbles}. In particular, we would like to resolve the question of game values for all-green games. As we have mentioned, it has proven difficult to determine these values when the underlying graph contains cycles. 

Through the use of computational software, in particular CGSuite~\cite{siegel2004combinatorial}, we have been able to find positions with many dyadic game values. It remains an open question whether or not there is a dyadic rational $\frac{a}{2^b}$ that is not the value of any position in \textsc{blocking pebbles}. 

\bibliography{pebblingbib}{}
\bibliographystyle{plain}

\end{document}